\documentclass[reprint,a4paper,11pt,reqno]{amsart}

\usepackage{amsmath}    % need for subequations
\usepackage{amsfonts}
\usepackage{amscd}
\usepackage[latin2]{inputenc}
\usepackage{t1enc}
\usepackage[mathscr]{eucal}
\usepackage{indentfirst}
\usepackage{graphicx}
\usepackage{graphics}
\usepackage{pict2e}
\usepackage{fancyhdr}
\numberwithin{equation}{section}
\usepackage[margin=2.9cm]{geometry}
\usepackage{epstopdf} 
\usepackage{hyperref}
\usepackage{inputenc}
\allowdisplaybreaks

\theoremstyle{plain}
\newtheorem{Th}{Theorem}%[section]
\newtheorem{Lemma}[Th]{Lemma}
\newtheorem{Cor}[Th]{Corollary}

\theoremstyle{definition}

\newtheorem{?}[Th]{Problem}

\hypersetup{
    colorlinks=true,
    linkcolor=blue,
    citecolor=blue,
    filecolor=magenta,      
    urlcolor=blue,
}

\begin{document}

%\date{\today}

\pagestyle{fancy}
\lhead{}
\chead{}
\rhead{\thepage}
\cfoot{}
\lfoot{}
\rfoot{}
\renewcommand{\headrule}{}

\title{Self-reciprocal functions and double Mordell integrals} 

\author{Martin Nicholson} 

\begin{abstract}  
The theory of self-reciprocal functions is applied to the study Mordell type integrals. We find two particular eigenfunctions of the double cosine Fourier transform and then use them to evaluate certain one- and two-dimensional Mordell type integrals in closed form. A reduction formula is given for a certain family of double Mordell integrals in terms of one-dimensional integrals.
\end{abstract}

%\date{\today}

\clearpage\maketitle
\thispagestyle{empty}
\vspace{-20pt}
\section{Introduction}

Mordell integrals are integrals of the form
\begin{equation}\label{mordell_int}
    \phi_\alpha(\theta)=\int\limits_0^\infty\frac{\cos(\pi \theta x)}{\cosh(\pi x)}\,e^{-\pi \alpha x^2}dx,\quad \psi_\alpha(\theta)=\int\limits_0^\infty\frac{\sin(\pi \theta x)}{\sinh(\pi x)}\,e^{-\pi \alpha x^2}dx.
\end{equation}
There is growing interest in multidimensional Mordell integrals in the literature. For example, multivariable Mordell integrals have been studied recently in connection with supersymmetric $U(N)$ Chern-Simons theories in quantum field theory \cite{rst}. Double Mordell integrals have been studied in connection with higher depth quantum modular forms and multiple Eichler integrals in \cite{bkm},\cite{bkm2}.

This paper is a continuation of the analysis that has been started in \cite{nicholson} where we have studied two-dimensional Mordell integrals using elementary methods, in particular the theory of self-reciprocal functions \cite{titchmarsh}. In particular, we have considered the double Mordell integral
\[
\Phi(\alpha,\beta,\gamma)=\int\limits_0^\infty \int\limits_0^\infty\frac{\cos(\pi \gamma xy)}{\cosh(\pi x)\cosh(\pi y)}\,e^{-\pi (\alpha x^2+\beta y^2)/2}\,dxdy,
\]
and proved the reduction formula in terms of the functions $\phi_\alpha(\theta)$
\[
(2n+1)\sqrt{{2}/{\alpha}}\,\Phi\left(\alpha^{-1}, (4n+2)^2\alpha,4n+2\right)=\left\{\phi_{\alpha}\left(\tfrac{i}{2}\right)\right\}^2+2\sum_{k=1}^n(-1)^{k}\phi_{\alpha}\Big(\tfrac{2n+2k+1}{4n+2}i\Big)\phi_{\alpha}\Big(\tfrac{2n-2k+1}{4n+2}i\Big),
\]
where $n\in\mathbb{N}_0$. Apriory, it is not obvious that such reduction formulas exist. One of the aims of this paper is to study double Mordell integrals reducible in terms of the functions $\psi_\alpha(\theta)$.

Among other formulas proved in \cite{nicholson} was the curios closed form evaluation of the integral
\begin{equation}\label{closed_form}
    \int\limits_0^\infty\tanh {(\pi x)}\tanh{(\alpha x)}\cos{(2\alpha x^2)}\,dx=0,\qquad \alpha>0.
\end{equation}
The fact that this integral converges can be seen by comparing it to the Fresnel integrals
\begin{equation}\label{fresnel}
    \int\limits_0^\infty\cos{(\alpha x^2)}\,dx=\int\limits_0^\infty\sin{(\alpha x^2)}\,dx=\sqrt{\frac{\pi}{8\alpha}},\qquad \alpha>0,
\end{equation}
using the asymptotics
\[
\tanh {(\pi x)}\tanh{(\alpha x)}=1+O(e^{-cx}), \quad c>0,\, x\to +\infty.
\]
Closed form evaluation of integrals that contain trigonometric functions of the argument $\alpha x^2$ and hyperbolic functions of both of the arguments $\pi x$ and $\alpha x$ have been known, for example \cite{berndt},\cite{ramanujan1}
\begin{equation}\label{closed_form3}
    \int\limits_0^\infty\frac{\cosh (\alpha x)}{\cosh (\pi x)}\cos{(\alpha x^2)}\,dx=\frac{1}{2}\cos\frac{\alpha}{4}, \qquad -\pi<\alpha<\pi.
\end{equation}
Generalization of this integral with interesting applications was given in \cite{glasser},\cite{glasser2},\cite{gm}
\begin{equation}\label{glasser}
    \int\limits_0^\infty\frac{ \cosh (\pi  x) \cosh (\alpha x)}{\cosh (2 \pi  x)+\cosh (2b)}\cos(\alpha x^2)\,dx=\frac{\cos\left(\frac{\alpha}{4}+\frac{\alpha b^2}{4\pi^2}\right)}{4 \cosh (b)},\qquad \alpha>0.
\end{equation}
What makes these formulas interesting is the fact that similar looking integrals do not always have closed form for all $\alpha>0$. This can be demonstrated by the integral \cite{GR},\cite{ramanujan2}
\begin{align}\label{series}
    \nonumber\int\limits_0^\infty\frac{e^{i\alpha x^2}}{\cosh(\pi x)}\cos(bx)\,dx=&\sum_{k=0}^\infty(-1)^ke^{-b\left(k+\frac12\right)-i\alpha\left(k+\frac12\right)^2}\\&+\sqrt{\frac{\pi}{\alpha}}\sum_{k=0}^\infty(-1)^ke^{-\frac{\pi b}{\alpha}\left(k+\frac12\right)+\frac{i\pi }{4}-\frac{ib^2}{4\alpha}+\frac{i\pi^2}{\alpha}\left(k+\frac12\right)^2},\quad \alpha>0,\, b>0.
\end{align}
One can notice that when $\alpha/\pi\in\mathbb{Q}$, the two series can be expressed in terms of finite sums, for example
\[
\int\limits_0^\infty\frac{e^{i\pi x^2}}{\cosh(\pi x)}\cos(bx)\,dx=\frac{e^{-\frac{\pi i}{4}}+ie^{-\frac{ib^2}{4\pi}}}{2\cosh\frac{b}{2}}.
\]
However, no apparent closed form exists for general $\alpha$.

One can notice that the poles of the integrand in formulas \ref{closed_form3}, \ref{glasser} form an arithmetic progression with the common difference $i$. However, unless $\alpha$ is a rational multiple of $\pi$, the integrand in \ref{closed_form} has two sets of poles that form arithmetic progressions with incommensurate common differences $i$ and $i\pi/\alpha$. For simplicity, in this paper, the integrals of the first type (equations \ref{closed_form3}, \ref{glasser}, \ref{series}) will be called type I, and of the second type with two incommensurate sets of poles (equation \ref{closed_form} and others to be considered in section \ref{type2}) will be called type II.

The fact that not all integrals of type II have closed form is demonstrated by the transformation formula \cite{nicholson}
\vspace{-5pt}
\begin{equation}\label{r2}
   \sqrt{2}\int\limits_0^\infty \frac{\cos (\alpha x^2)}{\cosh (\pi  x) \cosh (\alpha x)}\,dx=\int\limits_0^\infty  \frac{\cosh (\frac{\pi  x}{2})}{\cosh (\pi  x )}\cdot\frac{ \cosh (\frac{\alpha x}{2})}{\cosh (\alpha x)}\, dx, \qquad \alpha>0.
\end{equation}
It is evident that the right hand side of \ref{r2} can not have a closed form unless $\alpha/\pi\in\mathbb{Q}$. It is worth mentioning here that the function $\frac{\cosh (\frac{\pi  x}{2})}{\cosh (\pi  x )}$ is an eigenfunction of the cosine Fourier transform (up to rescaling). By applying Plancherel type argument to two different eigenfunctions of the cosine Fourier transform, Ramanujan derived transformation formulas for integrals of products of self-reciprocal hyperbolic functions (e.g., equation $10$ in \cite{ramanujan1}). 

Sometimes a type II integral can be expressed in terms of integrals of type I
\begin{equation}
    \int\limits_0^\infty\frac{\sin(\alpha x^2)}{\sinh(\pi x)\sinh(\alpha x)}\cos(bx)\,dx=\Bigg|\int\limits_0^\infty\frac{e^{i\alpha x^2}}{\cosh(\pi x)}\cos(bx)\,dx\Bigg|^2, \qquad \alpha>0, b>0.
\end{equation}
Although in \cite{nicholson}, this identity was proved for $b=0$, the proof easily can be extended to the case $b> 0$.

The organization of the paper is as follows. Following the same logic that have been used in \cite{nicholson}, we evaluate double Fourier transforms of certain functions of two variables in section \ref{auxiliary}. We find that the result of these Fourier transforms is the same function taken with minus sign, plus two terms with Dirac delta functions (that is, the functions considered are particular eigenfunctions of the double Fourier transform). The main difficulty of the analysis presented in this paper (sections \ref{auxiliary} and \ref{reduction}) will be the derivation of these Dirac delta function terms, which will be circumvented by proper regularizations of the singular integrals involved. Integrals with such delta function terms have not been considered in \cite{nicholson}. In section \ref{type2} we use these formulas to calculate two type II integrals in closed form. Two more type II integrals will follow by taking linear combinations of the first two. In section \ref{mordell}, several double Mordell integrals will be evaluated in close form. Reduction of a certain family of double Mordell integrals will be studied in section \ref{reduction}, similar to the reduction formula for $\Phi(\alpha,\beta,\gamma)$ mentioned above. Theorem of section \ref{reduction} generalizes double Mordell integral evaluations of section \ref{mordell}. In the Appendix, we give a new proof of the formula \ref{series} using Poisson summation formula.

\section{Two auxiliary integrals}\label{auxiliary}

Sokhotski--Plemelj formula \cite{m} states that
\[
\lim_{\varepsilon\to +0}\frac{1}{x\pm i\varepsilon}=\mathcal{P}\left(\frac{1}{x}\right)\mp i\pi\delta(x),
\]
where $\mathcal{P}$ denotes the Cauchy principal value, and $\delta$ is the Dirac delta function. Below we will use the following consequence of this formula
\begin{equation}\label{sp}
\lim_{\varepsilon\to +0}\frac{1}{\sinh(x\pm i\varepsilon)}=\mathcal{P}\left(\frac{1}{\sinh(x)}\right)\mp i\pi\delta(x),
\end{equation}
which can be checked using partial fractions expansions.

\begin{Lemma}\label{lemma1} For $a,b\in\mathbb{R}$
\begin{equation}\label{integral1}
    \frac{2}{\pi}\int\limits_0^\infty\int\limits_0^\infty\frac{\sin (xy)}{\tanh ( x) \tanh ( \pi y)}\cos (ax)\cos (by)\, dx dy=-\frac{\sin (ab)}{\tanh (\pi a) \tanh (b)}+\delta(a)+\pi\delta(b).
\end{equation}
\end{Lemma}

\begin{proof} Let
\[
I_{\varepsilon,\,\omega}(a,b)=\frac{2}{\pi}\int\limits_0^\infty\int\limits_0^\infty\frac{\cosh ((1-\varepsilon) x) \cosh ( \pi(1-\omega) y)}{\sinh (x) \sinh ( \pi y)}\sin (xy)\cos (ax)\cos (by)\, dx dy,
\]
where $0<\varepsilon<1$, $0<\omega<1$. The integral in the lemma is clearly divergent. It will be regularized as
\[
\lim_{\substack{\varepsilon\to +0\\ \omega\to +0}}I_{\varepsilon,\,\omega}(a,b).
\]
Writing $\sin(xy)\cos(by)=\frac{1}{2}\sin(y(x+b))+\frac{1}{2}\sin(y(x-b))$ and calculating the integral over $y$ using the formula $3.981.8$ from \cite{GR}
\begin{equation}\label{gradsteyn}
    \int\limits_0^\infty\frac{\cosh(\theta y)}{\sinh(\pi y)}\sin(a y)\, dy=\frac{1}{2}\cdot\frac{\sinh (a)}{\cosh(a)+\cos(\theta)}, \qquad 0<\theta<\pi,\, a>0,
\end{equation}
yields
\begin{align*}
    I_{\varepsilon,\,\omega}(a,b)&=\frac{1}{2\pi}\int\limits_0^\infty\frac{\cosh ((1-\varepsilon) x)}{\sinh (x)}\left(\frac{\sinh (x+b)}{\cosh (x+b)-\cos (\pi  \omega )}+\frac{\sinh (x-b)}{\cosh (x-b)-\cos (\pi  \omega )}\right)\cos (ax)\, dx\\
    &=\frac{1}{\pi}\int\limits_0^\infty\frac{\left(\cosh (x)-\cosh (b) \cos (\pi  \omega )\right)\cosh ((1-\varepsilon) x)}{(\cosh (x+b)-\cos (\pi  \omega )) (\cosh (x-b)-\cos (\pi  \omega ))}\cos (ax)\, dx\\
    &=\operatorname{Re}\left\{\frac{1}{2\pi}\int\limits_{-\infty}^\infty\frac{\cosh (x)-\cosh (b) \cos (\pi  \omega )}{(\cosh (x+b)-\cos (\pi  \omega )) (\cosh (x-b)-\cos (\pi  \omega ))}\,e^{(1-\varepsilon)x+iax} dx\right\}.
\end{align*}
Next, apply contour integration along a rectangular contour with vertices $(-R,0)$, $(R,0)$, $(R,2\pi i)$, $(-R,2\pi i)$, where $R\to\infty$. The integral over the line $\operatorname{Im}\, z=2\pi$ will be equal to the integral over the real axis times $-e^{2(1-\varepsilon)\pi i-2\pi a}$. After tedious but quite straightforward calculation using residue theorem one obtains
\[
I_{\varepsilon,\,\omega}(a,b)=\operatorname{Im}\left\{\frac{\sinh (\pi  a+i \pi  \epsilon -(i a-\epsilon +1) (b-i \pi  \omega ))}{2 \sinh (\pi  a+i \pi  \epsilon )\sinh (b-\pi  i \omega )}-\frac{\sinh (\pi  a+i \pi  \epsilon +(i a-\epsilon +1) (b+i \pi  \omega ))}{2 \sinh (\pi  a+i \pi  \epsilon )\sinh (b+i \pi  \omega )}\right\}.
\]
Hence
\[
\lim_{\substack{\varepsilon\to +0\\ \omega\to +0}}I_{\varepsilon,\,\omega}(a,b)=-\frac{\sin(ab)}{\tanh(\pi a)\tanh(b)}+\delta(a)+{\pi}\delta(b),
\]
from which the claim follows. \end{proof}

\begin{Lemma}\label{lemma2} For $a,b\in\mathbb{R}$
\[
\frac{2}{\pi}\int\limits_0^\infty\int\limits_0^\infty\frac{\sin (2xy)}{\tanh ( x) \tanh ( \pi y)}\cos (ax)\cos (by)\, dx dy=-\frac{\sin\frac{ab}{2}}{2\tanh\frac{\pi a}{2} \tanh \frac{b}{2}}+\delta(a)+\pi\delta(b).
\]
\end{Lemma}

\begin{proof} Let
\[
I_{\varepsilon,\,\omega}(a,b)=\frac{2}{\pi}\int\limits_0^\infty\int\limits_0^\infty\frac{\cosh ((1-\varepsilon) x) \cosh ( \pi(1-\omega) y)}{\sinh (x) \sinh ( \pi y)}\sin (2xy)\cos (ax)\cos (by)\, dx dy,
\]
where $0<\varepsilon<1$, $0<\omega<1$. The integral in the lemma is
\[
\lim_{\substack{\varepsilon\to +0\\ \omega\to +0}}I_{\varepsilon,\,\omega}(a,b).
\]
After simple calculation
\begin{align*}
    I_{\varepsilon,\,\omega}(a,b)&=\operatorname{Re}\left\{\frac{1}{\pi}\int\limits_{-\infty}^\infty\frac{\left(\cosh (2x)-\cosh (b) \cos (\pi  \omega )\right)\cosh(x)}{(\cosh (2x+b)-\cos (\pi  \omega )) (\cosh (2x-b)-\cos (\pi  \omega ))}\,e^{(1-\varepsilon)x+iax} dx\right\}.
\end{align*}
Using contour integration, or the formula from the proof of the previous lemma one finds
\[
I_{\varepsilon,\,\omega}(a,b)=\operatorname{Im}\left\{\frac{\sinh\frac{\pi  a+i \pi  \epsilon -(i a-\epsilon +1) (b-i \pi  \omega )}{2}}{4\sinh\frac{\pi  a+i \pi  \epsilon}{2}\sinh\frac{b-\pi  i \omega}{2}}-\frac{\sinh\frac{\pi  a+i \pi  \epsilon +(i a-\epsilon +1) (b+i \pi  \omega )}{2}}{4\sinh\frac{\pi  a+i \pi  \epsilon}{2}\sinh\frac{b+i \pi  \omega}{2}}\right\}.
\]
Hence
\[
\lim_{\substack{\varepsilon\to +0\\ \omega\to +0}}I_{\varepsilon,\,\omega}(a,b)=-\frac{\sin\frac{ab}{2}}{2\tanh\frac{\pi a}{2}\tanh\frac{b}{2}}+\delta(a)+{\pi}\delta(b),
\]
as required. \end{proof}

\section{Type II integrals}\label{type2}

\begin{Th}\label{int1} For $\alpha>0$
\[
\int\limits_0^\infty\frac{\sin(\alpha x^2)\,dx}{\tanh(\pi x)\tanh(\alpha x)}=\frac{1}{4}+\frac{\pi}{4\alpha}.
\]
\end{Th}

\begin{proof} In lemma \ref{lemma1}, put $b=\alpha a$ to obtain
\[
\frac{2}{\pi}\int\limits_0^\infty\int\limits_0^\infty\frac{\sin (xy)}{\tanh (x) \tanh (\pi y)}\cos (ax)\cos (\alpha a y)\, dx dy=-\frac{\sin (\alpha a^2)}{\tanh (\pi a) \tanh (\alpha a)}+\left(1+\frac{\pi}{\alpha}\right)\delta(a).
\]
Integrating with respect to $a$ from $0$ to $\infty$ using the formulas
\[
\int\limits_0^\infty \delta(a)\,da=\frac{1}{2},
\]
\[
\frac{2}{\pi}\int\limits_0^\infty \cos (ax)\cos (\alpha a y)\,da=\delta(x-\alpha y)+\delta(x+\alpha y),
\]
one finds
\begin{align*}
    \int\limits_0^\infty\int\limits_0^\infty\frac{\sin (xy)}{\tanh (x) \tanh ( \pi y)}&\left(\delta(x-\alpha y)+\delta(x+\alpha y)\right)\, dx dy\\
    &=-\int\limits_0^\infty\frac{\sin (\alpha a^2)}{\tanh (\pi a) \tanh (\alpha a)}\, da+\frac{1}{2}\left(1+\frac{\pi}{\alpha}\right).
\end{align*}
After calculating the integral over $x$ one finds the same integral both on the right and left hand side. Thus
\[
2\cdot\int\limits_0^\infty\frac{\sin (\alpha a^2)}{\tanh (\pi a) \tanh (\alpha a)}\, da=\frac{1}{2}\left(1+\frac{\pi }{\alpha}\right),
\]
as required.
\end{proof}

\begin{Th}\label{int2} For $\alpha>0$
\[
\int\limits_0^\infty\frac{\sin(2\alpha x^2)\,dx}{\tanh(\pi x)\tanh(\alpha x)}=\frac{1}{4}+\frac{\pi}{4\alpha}.
\]
\end{Th}
\begin{proof}
Proof follows from lemma \ref{lemma2} along the same lines as in the proof of the previous theorem.
\end{proof}

\begin{Cor}\label{int3} For $\alpha>0$
\[
\int\limits_0^\infty\frac{\tanh(\alpha x)}{\tanh(\pi x)}\sin(2\alpha x^2)\,dx=\frac{1}{4}.
\]
\end{Cor}

\begin{proof}
Use the elementary identity
\begin{equation}\label{elementary}
    2\coth(2x)-\coth(x)=\tanh(x)
\end{equation}
and the previous two theorems.
\end{proof}

Note that when $\alpha=\pi$, the formula in Theorem \ref{int3} reduces to one of the Fresnel integrals \ref{fresnel}.

\begin{Cor} For $\alpha>0$
\[
\int\limits_0^\infty\frac{\sin(\alpha x^2)\,dx}{\sinh(\pi x)\tanh(\alpha x)}=\frac{1}{4}.
\]
\end{Cor}

\begin{proof}
Use the elementary identity
\begin{equation}\label{elementary2}
    \coth(x)-\coth(2x)=\frac{1}{\sinh(2x)}
\end{equation}
and theorems \ref{int1} and \ref{int2}.
\end{proof}

\section{Double Mordell integrals}\label{mordell}

\begin{Lemma} Let $f(x)$ be an eigenfunction of the cosine Fourier transform, and $\alpha\beta=\pi$ or $\alpha\beta=\frac{\pi}{2}$. Then
\[
\int\limits_0^\infty\int\limits_0^\infty\frac{\sin (x y) }{\tanh (\alpha  x) \tanh (\beta y)}\,f(x)f(y)\,dxdy=\frac{\pi^{3/2}}{4\sqrt{2}}\left(\alpha^{-1}+\beta^{-1}\right)\left(f(0)\right)^2.
\]
\end{Lemma}
\begin{proof} The fact that $f(x)$ is an eigenfunction of the cosine Fourier transform means that
\[
\sqrt{\frac{2}{\pi}}\int\limits_0^\infty f(x)\cos(bx)\,dx=f(b).
\]
Multiplying \ref{integral1} by $f(sa)f(tb)$ and integrating with respect to $a$ and $b$ from $0$ to $\infty$ one obtains
\begin{align*}
    &\frac{1}{st}\int\limits_0^\infty\int\limits_0^\infty\frac{\sin (xy)}{\tanh ( x) \tanh ( \pi y)}f(x/s)f(y/t)\, dx dy\\&=-\int\limits_0^\infty\int\limits_0^\infty\frac{\sin (ab)}{\tanh (\pi a) \tanh (b)}f(sa)f(tb)\,dadb+\frac{1}{2}\sqrt{\frac{\pi}{2}}\,(t^{-1}+\pi s^{-1})\left(f(0)\right)^2.
\end{align*}
It is not hard to notice that when $st=1$ both double integrals are equal to each other. Redefining the parameters according to $s=\alpha$, $\pi t=\beta$, we complete the proof of the case with $\alpha\beta=\pi$.
The case $\alpha\beta=\frac{\pi}{2}$ is derived from lemma \ref{lemma2} in a similar manner.
\end{proof}

\begin{Th}\label{2dmordell} If $\alpha\beta=2\pi$ or $\alpha\beta=\pi$, then
\[
\int\limits_0^\infty\int\limits_0^\infty\frac{\sin (2  x y) }{\tanh (\alpha  x) \tanh (\beta y)}\,e^{-x^2-y^2}dxdy=\frac{\pi^{3/2}}{8}\left(\alpha^{-1}+\beta^{-1}\right).
\]
\end{Th}

\begin{proof} This is direct consequence of the previous theorem and the fact that $f(x)=e^{-x^2/2}$ is an eigenfunction of the cosine Fourier transform.
\end{proof}

\begin{Cor} If $\alpha\beta=\pi$, then
\[
\int\limits_0^\infty\int\limits_0^\infty\frac{\tanh (\alpha  x)}{\tanh (\beta y)}\,\sin (2x y) \,e^{-x^2-y^2}dxdy=\frac{\sqrt{\pi}}{8}{\alpha}.
\]
\end{Cor}

\begin{proof}
Use identity \ref{elementary} and the previous theorem.
\end{proof}

\begin{Cor} If $\alpha\beta=2\pi$, then
\[
\int\limits_0^\infty\int\limits_0^\infty\frac{\sin (2x y) }{\tanh (\alpha  x) \sinh (\beta y)}\,e^{-x^2-y^2}dxdy=\frac{\sqrt{\pi}}{16}{\alpha}.
\]
\end{Cor}

\begin{proof} 
Use identity \ref{elementary2} and theorem \ref{2dmordell}.
\end{proof}

We note that \ref{closed_form} could be derived from the double Fourier transform of the function
\[
\tanh {(x)}\tanh{(\pi y)}\cos{(2xy)},
\]
though the proof in \cite{nicholson} used different Fourier transforms. There is also corresponding evaluation for double Mordell integral. 

Unfortunately, the theory developed in this paper becomes too cumbersome for more complicated functions. It would be interesting to find different proofs of the formulas in section \ref{type2}. Such proofs could be useful in finding more type II integrals that can be evaluated in closed form. 

\section{Reduction formula for a certain family of double Mordell integrals}\label{reduction}
Consider the integral
\[
\Psi(\alpha,\beta,\gamma)=\int\limits_0^\infty \int\limits_0^\infty\frac{\sin( \pi \gamma xy)}{\tanh (\pi x)\tanh(\pi y)}\,e^{-\pi
(\alpha x^2+\beta y^2)/2}\,dxdy.
\]
Is there a combination of parameters $\alpha$, $\beta$, $\gamma$ such that $\Psi(\alpha,\beta,\gamma)$ reduces to a sum of products of one-dimensional Mordell integrals $\psi_\alpha(\theta)$ defined in \ref{mordell_int}. The answer to this question will be given below. We start from proving a general three-parameter transformation formula for $\Psi(\alpha,\beta,\gamma)$. A similar formula has been proved for $\Phi(\alpha,\beta,\gamma)$ using somewhat different notation in \cite{nicholson}.
\begin{Lemma}\label{psi}
For $\alpha,\beta,\gamma>0$, we have
\begin{equation}\label{transformation}
    \Psi(\alpha,\beta,\gamma)=\frac{2}{\sqrt{\alpha\beta+\gamma^2}}\,\Psi\left(\frac{4\alpha}{\alpha\beta+\gamma^2},\frac{4\beta}{\alpha\beta+\gamma^2},\frac{4\gamma}{\alpha\beta+\gamma^2}\right).
\end{equation}
\end{Lemma}

\begin{proof}
Define according to formula $(14.4.1)$ in \cite{berndt}
\[
F_\alpha(\theta)=\int\limits_0^\infty\frac{\sin(\pi \theta x)}{\tanh(\pi x)}\,e^{-\pi \alpha x^2}dx.
\]
It is known that $F_\alpha(\theta)$ satisfies the transformation formula (\cite{berndt}, Entry 14.4.1)
\begin{equation*}
    F_\alpha(\theta)=\frac{-i}{\sqrt{\alpha}}\,e^{-\pi\theta^2/(4\alpha)}F_{1/{\alpha}}(i\theta/\alpha).
\end{equation*}
\begin{align}\label{transform}
\nonumber\Psi(\alpha,\beta,\gamma)&=\int\limits_0^\infty\frac{e^{-\pi\beta y^2/2}}{\tanh(\pi y)}\,F_{\alpha/2}(\gamma y)\,dy\\
&=-i\sqrt{\frac{2}{\alpha}}\int\limits_0^\infty\frac{e^{{-{\pi}\left(\beta+\gamma^2/{\alpha}\right)y^2/2}}}{\tanh(\pi y)}\,F_{2/\alpha}(2i\gamma y/\alpha)\,dy\nonumber\\
&=\sqrt{\frac{2}{\alpha}}\int\limits_{0}^\infty\int\limits_{0}^\infty\frac{e^{-2\pi x^2/\alpha-{\pi}\left(\beta+\gamma^2/{\alpha}\right)y^2/2}}{\tanh(\pi x)\,\tanh(\pi y)}\,\sinh\Big(\frac{2\pi \gamma xy}{\alpha}\Big)\,dxdy\\
&=-i\sqrt{\frac{2}{\alpha}}\int\limits_0^\infty\frac{e^{-2\pi x^2/\alpha}}{\tanh(\pi x)}\,F_{(\beta+\gamma^2/\alpha)/2}(2i\gamma x/\alpha)\,dx\nonumber\\
&=\sqrt{\frac{4}{\alpha\beta+\gamma^2}}\int\limits_0^\infty\frac{e^{-2\pi\beta x^2/(\alpha\beta+\gamma^2)}}{\tanh(\pi x)}\,F_{2/(\beta+\gamma^2/\alpha)}(4\gamma x/(\alpha\beta+\gamma^2)\,dx.\nonumber
\end{align}
One can easily see that this is equal to the right hand side of \ref{transformation}, as required.
\end{proof}

\begin{Th} \label{psi2}
Let $n\in\mathbb{N}$ and $\alpha>0$. Then
\begin{equation}\label{psi3}
    \frac{8}{\sqrt{n}}\Psi(2\alpha/n,2\alpha^{-1}/n,2/n)=\sqrt{{32n}}\,\Psi(\alpha^{-1}n, \alpha n,n)=\sqrt{\alpha}+\frac{1}{\sqrt{\alpha}}-\sqrt{\alpha}\sum_{k=1}^{n-1}\left\{\psi_{\alpha/n}\left(\tfrac{n-2k}{n}i\right)\right\}^2.
\end{equation}
\end{Th}

\begin{proof} The first equality follows from lemma \ref{psi}. Thus one has to established only the second equality. From \ref{transform}
\[
\Psi(\alpha^{-1}, n^2\alpha,n)=\sqrt{{2}{\alpha}}\int\limits_{0}^\infty\int\limits_{0}^\infty\frac{e^{-2\pi\alpha x^2-\pi\alpha n^2y^2}}{\tanh(\pi x)\,\tanh(\pi y)}\,\sinh (2\pi \alpha n xy)\,dxdy.
\]
We want to extend the integration over the whole $(x,y)$ plane, but to do this one has to regularize the integral, which will be done using Cauchy principal values as follows
\begin{align*}
    \Psi(\alpha^{-1}, n^2\alpha,n)&=\sqrt{\frac{\alpha}{8}}\cdot\mathcal{P}\int\limits_{-\infty}^\infty\int\limits_{-\infty}^\infty\frac{e^{-2\pi\alpha x^2-\pi\alpha n^2y^2}}{\tanh(\pi x)\,\tanh(\pi y)}\,e^{2\pi \alpha n xy}\,dxdy\\
    &=\sqrt{\frac{\alpha}{8}}\cdot\mathcal{P}\int\limits_{-\infty}^\infty\int\limits_{-\infty}^\infty\frac{e^{-2\pi\alpha (x- ny/2)^2-\pi\alpha n^2y^2/2}}{\tanh(\pi x)\,\tanh(\pi y)}\,dxdy.
\end{align*}
Further calculations require a change of variables. However, this form of the integral is not well suited for application of a change of variables, so we will give another regularization that follows from the above formula using  Sokhotski--Plemelj theorem \ref{sp}
\[
\Psi(\alpha^{-1}, n^2\alpha,n)=\sqrt{\frac{\alpha}{8}}\int\limits_{-\infty}^\infty\int\limits_{-\infty}^\infty e^{-2\pi\alpha (x- ny/2)^2-\pi\alpha n^2y^2/2}\left(\tfrac{1}{\tanh(\pi x+i\varepsilon)}+i\delta(x)\right)\left(\tfrac{1}{\tanh(\pi y+i\omega)}+i\delta(y)\right)\,dxdy,
\]
where $\varepsilon\to +0$ and $\omega\to +0$. Since 
\[
\int\limits_{-\infty}^\infty\int\limits_{-\infty}^\infty e^{-2\pi\alpha (x- ny/2)^2-\pi\alpha n^2y^2/2}\left(\tfrac{1}{\tanh(\pi y+i\omega)}+i\delta(y)\right)\delta(x)\,dxdy=0,
\]
and
\[
\int\limits_{-\infty}^\infty\int\limits_{-\infty}^\infty e^{-2\pi\alpha (x- ny/2)^2-\pi\alpha n^2y^2/2}\tfrac{1}{\tanh(\pi x+i\varepsilon)}\,i\delta(y)\,dxdy=1,
\]
one can simplify the above expression as
\[
\Psi(\alpha^{-1}, n^2\alpha,n)=\sqrt{\frac{\alpha}{8}}+\sqrt{\frac{\alpha}{8}}\int\limits_{-\infty}^\infty\int\limits_{-\infty}^\infty \frac{e^{-2\pi\alpha (x- ny/2)^2-\pi\alpha n^2y^2/2}}{\tanh(\pi x+i\varepsilon)\,\tanh(\pi y+i\omega)}\,dxdy.
\]
Now, after the change of variables $x\to x+ny/2$, one obtains
\begin{align*}
    \Psi(\alpha^{-1}, n^2\alpha,n)&=\sqrt{\frac{\alpha}{8}}+\sqrt{\frac{\alpha}{8}}\int\limits_{-\infty}^\infty\int\limits_{-\infty}^\infty \frac{e^{-2\pi\alpha x^2-\pi\alpha n^2y^2/2}}{\tanh(\pi x+\pi ny/2+i\varepsilon)\,\tanh(\pi y+i\omega)}\,dxdy\\
    &=\sqrt{\frac{\alpha}{8}}+\sqrt{\frac{\alpha}{32}}\int\limits_{-\infty}^\infty\int\limits_{-\infty}^\infty \frac{e^{-2\pi\alpha x^2-\pi\alpha n^2y^2/2}}{\tanh(\pi y+i\omega)}\left(\tfrac{1}{\tanh(\pi x+\pi ny/2+i\varepsilon)}-\tfrac{1}{\tanh(\pi x-\pi ny/2-i\varepsilon)}\right)\,dxdy\\
    &=\sqrt{\frac{\alpha}{8}}-\sqrt{\frac{\alpha}{32}}\int\limits_{-\infty}^\infty\int\limits_{-\infty}^\infty\frac{\sinh(\pi ny+2i\varepsilon)}{\sinh(\pi y+i\omega)}\frac{e^{-2\pi\alpha x^2-\pi\alpha n^2y^2/2}\cosh(\pi y+i\omega)\,dxdy}{\sinh(\pi x+\pi ny/2+i\varepsilon)\sinh(\pi x-\pi ny/2-i\varepsilon)}.
\end{align*}
If $2\varepsilon=\omega n$, then one can write
\begin{align*}
    \frac{\sinh(\pi ny+2i\varepsilon)}{\sinh(\pi y+i\omega)}\cosh(\pi y+i\omega)&=\cosh(\pi y+i\omega)\sum_{k=0}^{n-1}e^{(n-1-2k)(\pi y+i\omega)}\\
    &=\cosh(\pi ny+in\omega)+\sum_{k=1}^{n-1}e^{(n-2k)(\pi y+i\omega)}.
\end{align*}
Putting $\omega=0$ in the numerator and after the change of variables $\xi=x+n y/2,~\eta=x-ny/2$, one obtains
\begin{align*}
    \Psi(\alpha^{-1}, n^2\alpha,n)&=\sqrt{\frac{\alpha}{8}}-\sqrt{\frac{\alpha}{32n^2}}\int\limits_{-\infty}^\infty\int\limits_{-\infty}^\infty\Big(\cosh\big(\pi(\xi-\eta)\big)+\sum_{k=1}^{n-1}e^{\pi (\xi-\eta)\frac{n-2k}{n}}\Big)\frac{e^{-\pi\alpha( \xi^2+\eta^2)}\,d\xi d\eta}{\sinh(\pi\xi+i\varepsilon)\sinh(\pi \eta-i\varepsilon)}.
\end{align*}
After simplifying the integrals using the relations
\[
\int\limits_{-\infty}^\infty\int\limits_{-\infty}^\infty\frac{\cosh\left(\pi(\xi-\eta)\right)}{\sinh(\pi\xi+i\varepsilon)\sinh(\pi \eta-i\varepsilon)}\,e^{-\pi\alpha( \xi^2+\eta^2)}\,d\xi d\eta=\int\limits_{-\infty}^\infty\int\limits_{-\infty}^\infty\left\{\delta(\xi)\delta(\eta)-1\right\} e^{-\pi\alpha( \xi^2+\eta^2)}\,d\xi d\eta=1-\alpha^{-1},
\]
\[
\int\limits_{-\infty}^\infty\int\limits_{-\infty}^\infty \frac{e^{\pi (\xi-\eta)\frac{n-2k}{n}}}{\sinh(\pi\xi+i\varepsilon)\sinh(\pi \eta-i\varepsilon)}\,e^{-\pi\alpha( \xi^2+\eta^2)}\,d\xi d\eta=1+\left\{\psi_{\alpha}\left(\tfrac{n-2k}{n}i\right)\right\}^2,
\]
we finally come to
\[
\Psi(\alpha^{-1}, n^2\alpha,n)=\sqrt{\frac{\alpha}{32}}\left(1+\frac{1}{\alpha n}-\frac{1}{n}\sum_{k=1}^{n-1}\left\{\psi_{\alpha}\left(\tfrac{n-2k}{n}i\right)\right\}^2\right),
\]
which is equivalent to the second equality in \ref{psi3}, as required.
\end{proof} 

The corollary below is an immediate consequence of Theorem \ref{psi2}.
\begin{Cor}\label{quadratic}
If $\alpha\beta=1$, and $n\in\mathbb{N}$, then
\[
\sqrt{\alpha}\sum_{k=1}^{n-1}\left\{\psi_{\alpha/n}\left(\tfrac{n-2k}{n}i\right)\right\}^2=\sqrt{\beta}\sum_{k=1}^{n-1}\left\{\psi_{\beta/n}\left(\tfrac{n-2k}{n}i\right)\right\}^2.
\]
\end{Cor}
\noindent We note that in principle this identity can be derived from linear relations between one-dimensional Mordell integrals. This is because $\psi_{\beta/n}\left(\tfrac{n-2k}{n}i\right)$ for any $1\le k\le n-1$ can be written as a linear sum of $n-1$ integrals $\psi_{\alpha/n}\left(\tfrac{n-2k}{n}i\right)$ (for analogous formulas see \cite{cais}). The transformation matrix from one bases to another is orthogonal, thus preserves the diagonal quadratic form.

As an illustration of \ref{psi3}, consider some simple cases. When $n=1$ or $n=2$, then \ref{psi3} gives Theorem \ref{2dmordell}. Thus Theorem \ref{psi2} is a generalization of Theorem \ref{2dmordell}. When $n=3$ or $n=4$, one gets
\[
4\sqrt{6}\int\limits_0^\infty \int\limits_0^\infty\frac{\sin(3\pi xy)}{\tanh (\pi x)\tanh(\pi y)}\,e^{-3\pi
(\alpha x^2+\alpha^{-1}y^2)/2}\,dxdy=\sqrt{\alpha}+\frac{1}{\sqrt{\alpha}}+18\sqrt{\alpha}\left\{\int\limits_0^\infty \frac{e^{-3\pi\alpha x^2}}{2\cosh({2\pi x})+1}\, dx\right\}^2,
\]
\[
8\sqrt{2}\int\limits_0^\infty \int\limits_0^\infty\frac{\sin(4\pi xy)}{\tanh (\pi x)\tanh(\pi y)}\,e^{-2\pi(\alpha x^2+\alpha^{-1}y^2)}\,dxdy=\sqrt{\alpha}+\frac{1}{\sqrt{\alpha}}+2\sqrt{\alpha}\left\{\int\limits_0^\infty \frac{e^{-\pi\alpha x^2}}{\cosh({\pi x})}\, dx\right\}^2.
\]

\section*{Appendix: Proof of \ref{series}}

Ramanujan gave a proof of \ref{series} using Laplace transform \cite{ramanujan2}. Proof using contour integration can be found in \cite{mordell}. The proof below is based on Poisson summation formula.  

\begin{proof}
Using the partial fractions expansion
\[
\frac{1}{\cosh(\pi x)}=\frac{1}{\pi}\sum_{k=0}^\infty\frac{(-1)^k(2k+1)}{x^2+\left(k+\frac{1}{2}\right)^2},
\]
and integrating termwise one obtains
\[
I(\alpha)=\frac{1}{2\pi}\sum_{k=0}^\infty(-1)^k(2k+1)\,I_k(\alpha),
\]
where 
\[
I_k(\alpha)=\int\limits_{-\infty}^\infty\frac{e^{i\alpha x^2+ibx}}{x^2+\left(k+\frac{1}{2}\right)^2}\, dx.
\]
Integrals of this form can be reduced to error function. $I_k(\alpha)$ satisfies the following differential equation
\[
I_k'(\alpha)+i\left(k+\frac{1}{2}\right)^2I_k(\alpha)=i\int\limits_{-\infty}^\infty e^{i\alpha x^2+ibx}\,dx=\sqrt{\frac{\pi}{\alpha}}\,e^{\frac{3i\pi }{4}-\frac{ib^2}{4\alpha}},
\]
with the initial condition $I_k(0)=\frac{2\pi}{2k+1}\,e^{-b\left(k+\frac{1}{2}\right)}$. One can check by direct calculation (using integration by parts) that the solution is given by 
\[
I_k(\alpha)=\frac{2\pi}{2k+1}\,e^{-b\left(k+\frac{1}{2}\right)-i\alpha\left(k+\frac{1}{2}\right)^2}+\frac{4\sqrt{\pi}}{2k+1}\,e^{\frac{i\pi }{4}-\frac{ib^2}{4\alpha}}\int\limits_0^\infty e^{iy^2-\frac{yb}{\sqrt{\alpha}}}\sin\left(y\sqrt{\alpha}(2k+1)\right)\, dy.
\]
Thus
\begin{align*}
    I(\alpha)&=\sum_{k=0}^\infty(-1)^ke^{-b\left(k+\frac12\right)-i\alpha\left(k+\frac12\right)^2}\\
    &+2e^{\frac{i\pi }{4}-\frac{ib^2}{4\alpha}}\sqrt{\frac{\pi}{\alpha}}\sum_{k=0}^\infty(-1)^k\int\limits_0^\infty e^{\frac{i\pi^2}{\alpha}y^2-\frac{\pi b}{\alpha}y}\sin\left(\pi(2k+1)y\right)\, dy.
\end{align*}
Next, we apply Poisson summation formula in the form \cite{titchmarsh}
\[
\sum_{k=0}^\infty(-1)^k\int\limits_0^\infty f(y)\sin\left(\pi(2k+1)y\right)\, dy=\frac{1}{2}\sum_{k=-\infty}^\infty(-1)^k f\left(k+\tfrac{1}{2}\right).
\]
This completes the proof.
\end{proof}

\end{document}